\documentclass[a4paper,11pt]{article}
\usepackage{lmodern}
\usepackage[T1]{fontenc}
\usepackage{textcomp} 
\usepackage[latin1]{inputenc}
\usepackage{amsfonts,amsmath}
\usepackage{geometry}
\usepackage[english]{babel}
\usepackage{color}
\usepackage[normalem]{ulem}
\usepackage{mathdots}
\usepackage{amssymb}
\usepackage{amsthm}
\usepackage{graphicx}
\usepackage{leftidx}

\title{The $3n+1$ problem: a partition of interest}
\author{Maarten Wensink \\ maartenwensink@hotmail.com}
\date{\today}
\begin{document}
\theoremstyle{plain}
\newtheorem{theorem}{Theorem}
\newtheorem{corollary}[theorem]{Corollary}
\newtheorem{lemma}[theorem]{Lemma}
\newtheorem{proposition}[theorem]{Proposition}
\theoremstyle{definition}
\newtheorem{definition}[theorem]{Definition}
\newtheorem{example}[theorem]{Example}
\newtheorem{conjecture}[theorem]{Conjecture}
\theoremstyle{remark}
\newtheorem*{remark}{Remark}

\maketitle
\newcommand{\N}{\mathbb{N}}
\newcommand{\0}{\mathbb{N}_0}
\newcommand{\D}{\mathbb{O}}

\tableofcontents

\section{Abstract}
A mapping conjugate to the Collatz mapping seems to imply that $\N=\{1,2,3,\ldots\}$ is partitioned in a trivial loop $\{1\}$ and `strings' that are ordered subsets of $\{\N \setminus 1\}$ that run from an element of $\{2+3\0\}$ to an element of $\{3+4\0\}$ ($\0=0 \cup \N$). In particular, this means that all trajectories except for the trivial loop go through an element of $\{3+4\0\}$ ($\{5+8\0\}$ for the original mapping). I give reasons for this conjecture. Next, I note that the 3n+1 numbers and the 3n+3 numbers are the only numbers from the generalization $3n+p, p \in \{\ldots,-3,-1,1,3,\ldots\}$ for which such a partition seems to exist. Suspiciously, these are also the only members for which the conjecture (reduction to the trivial loop) seems to hold.

\section{Preliminary analysis}
The Collatz conjecture / 3n+1 problem posits that recursive application of the mapping
\begin{equation}
\label{Collatz}
	C(n):= 	\begin{cases}
					n/2 &\mbox{if n is even}; \\ 
					3n+1 & \mbox{if n is odd}, 
			\end{cases} 
\end{equation}
\newline
on any natural number $n \in \N$, $\N=1,2,3,\ldots$, eventually leads to 1, after which the cycle $\{4,2,1\}$ is repeated indefinitely (Pickover 2009, Lagarias 2010). The $3n+1$ problem has mystified mathematicians for decades. If true, how could such a simple rule reduce all natural numbers to 1? 

Instead of the Collatz map (equation (\ref{Collatz})) we might study the accelerated Collatz map that sends odd positive integers to odd positive integers:
\begin{equation}
\widetilde{C}(n):=\frac{3n+1}{2^j},
\label{AccelCollatz}
\end{equation}
where $2^j$ is the largest power of 2 that divides $3n+1$, with $n \in \D$, where $\D=\{1,3,5,\ldots\}$, the odd positive integers. We might enumerate the odd positive integers, which gives just $\N$, such that the $i^th$ element of $\N$ corresponds to the $i^{th}$ element of $\D$. Thus, the odd positive integers are enumerated: $g(1)=1$, $g(3)=2$, $g(5)=3$, et cetera. For efficient reference and to avoid confusion, I put these enumerated positive integers between brackets: $g(1)=[1]$, $g(3)=[2]$, $g(5)=[3]$, et cetera. When an entire equation is put between brackets, the entire equation refers to the enumerated positive integer space. In addition, I use $x$ to refer to enumerated positive integers, whereas $n$ refers to those positive integers themselves: $g(n)=[x]$.

$\D$ can be transformed to $\N$ by $g: \D \rightarrow \N$ such that 
\begin{equation}
g(n):=\frac{n+1}{2}.
\label{g}
\end{equation}
The inverse of $g$ is $g^{-1}: \N \rightarrow \D$ such that
\begin{equation}
g^{-1}([x])=2[x]-1=n.
\label{g-1}
\end{equation}

Conjugating $\widetilde{C}$ through $g$ yields mapping
\begin{equation}
F([x])=g(\widetilde{C}(g^{-1}([x]))),
\end{equation}
which is $F:\N \rightarrow 1+3\0 \cup 3+3\0$, such that
\begin{equation}
\begin{cases}
\label{F}
F(E^n([2+2m]))=[3+3m]|(m,n) \in \0^2, \\
F(E^n([1+4m]))=[1+3m]|(m,n) \in \0^2,
\end{cases}
\end{equation}
with $E: [\N \rightarrow 3+4\0]$ (explanation below) such that
\begin{eqnarray}
\label{E}
E([x]):&=&4[x]-1, \\
\label{E0}
E^0([x]):&=& [x], \\
\label{En}
E^n([x]):&=&E(E^{n-1}([x])).
\end{eqnarray}

For completeness, I prove $F$, including its domain ($[\N]$) and range ($[1+3\0 \cup 3+3\0]$), in the Appendix. The steps involved are straightforward and mimic known results for the original Collatz map. See Wirsching 1998.

\begin{remark}
Notice that the trivial cycle $\{4,2,1\}$ now becomes the trivial loop [1], since 4 and 2 are even and 1 is the first odd natural number.
\end{remark}

It is essential to thoroughly understand $E$. $E$ has the property that
\begin{equation}
F(E^j([x]))=F(E^k([x])) \; \forall (j,k) \in \0^2, \; [x \in \N]
\end{equation}
Thus, $E$ indicates which members of $[\N]$ have the same image as $[x \in \N]$ under $\widetilde{C}$ and are therefore `equivalent'; $E$ is not itself the mapping under $\widetilde{C}$. It presents itself as a function, because which elements of $[\N]$ are equivalent to $[x]$ clearly depends on $[x]$. The domain of $E$ is $[\N]$, while its range is $[E(\N)=3+4\0]$. Hence, every $[x \in \N]$ has infinitely many higher equivalents, whereas members of $[3+4\0]$ have at least one lower equivalent. Similarly, $[E(3+4\0)=11+16\0]$ have at least two lower equivalents, and so forth. Phrases like ``taking equivalents" mean ``applying $E$". Equivalents could be taken of a single $[x]$, forming $[E(x)]$, then $[E(E(x))=E^2(x)]$, and so forth.

Equivalents could also be taken of a set. For instance, the set $[2+2\0]$ has equivalents $[E(2+2\0)=7+8\0]$, $[E^2(2+2\0)=E(7+8\0)=27+32\0]$ et cetera, while the set $[1+4\0]$ has equivalents $E[(1+4\0)=3+16\0]$, $[E^2(1+4\0)=E(3+16\0)=11+64\0]$ et cetera. Indeed, the domain of $F$ as presented in equation (\ref{F}) follows from taking these equivalents of sets. It is a partition of $[\N]$ that consists of two collections of equivalent sets: those that map to $[3+3\0]$ are equivalents of $[2+2\0]$, while those that map to $[1+3\0]$ are equivalents of $[1+4\0]$. The union of all those equivalents of $[2+2\0]$ and $[1+4\0]$, including $[2+2\0]$ and $[1+4\0]$ themselves, is $[\N]$. Thus, $F$ suggests to write $[\N]$ as $[2+2\0 \cup 1+4\0 \cup E(2+2\0) \cup E(1+4\0) \cup E^2(2+2\0) \cup E^2(1+4\0) \cup \ldots = 2+2\0 \cup 1+4\0 \cup 7+8\0 \cup 3+16\0 \cup 27+32\0 \cup 11+64\0 \cup \ldots]$, as this is the way the domain of $F$, which is $[\N]$, presents itself (Appendix).

We could consider restrictions of $F$ that each pertain to one of these subsets of $[\N]$ into which the domain of $F$ is naturally divided. This gives $F|_{[2+2\0]}$ being the part of $F$ that maps $[2+2\0]$ (i.e., $[2+2\0]$ is its domain) to $[3+3\0]$ (i.e., $[3+3\0]$ is its range), $F|_{[1+4\0]}$ being the part of $F$ that maps $[1+4\0]$ to $[1+3\0]$, $F|_{E([2+2\0])}=F|_{[7+8\0]}$ being the part of $F$ that maps $[7+8\0]$ to $[3+3\0]$, $F|_{E([1+4\0])}=F|_{[3+16\0]}$ being the part of $F$ that maps $[3+16\0]$ to $[1+3\0]$, et cetera. Conform with the fact that natural numbers are distributed uniformly modulus 2, we find that of any and all two consecutive elements of $[\N]$, exactly one is mapped through $F|_{[2+2\0]}$, that of any and all four consecutive elements of $[\N]$, exactly one is mapped through $F|_{[1+4\0]}$, that of any and all eight consecutive elements of $[\N]$, exactly one is mapped through $F|_{[7+8\0]}$, and so forth. The distance in $[\N]$ between any two consecutive elements of $[\N]$ mapped through $F|_{[2+2\0]}$ is $2^1$, the distance in $[\N]$ between any two consecutive elements of $[\N]$ mapped through $F|_{[1+4\0]}$ is $2^2$, while the distance in $[\N]$ between any two consecutive elements of $[\N]$ mapped through $F|_{[7+8\0]}$ is $2^3$, and so forth. These distances, which I call `intervals', are all powers of 2. I use these powers of 2 to refer to the restrictions of $F$ introduced above: $F_1:=F|_{[2+2\0]}$, $F_2:=F|_{[1+4\0]}$, $F_3:=F|_{[7+8\0]}$, and so forth. Just like I use $[x]$ to refer to any element of $[\N]$ without singling out any element of $[\N]$ in particular, so I use $F_z$, $z \in \N$ to refer to any of these restrictions, without specifying which: if I specify $z=1$, I refer to $F_1=F|_{[2+2\0]}$, if I specify $z=2$ I refer to $F_2=F|_{[1+4\0]}$, and so forth. Consecutive elements of $[\N]$ that are mapped through some $F_z$ are found in $[\N]$ at intervals of $2^z$, and of any and all $2^z$ consecutive  elements of $[\N]$, exactly one is mapped through $F_z$.

The property that of any and all $2^z$ consecutive elements, exactly one is mapped through $F_z$ is maintained over an arbitrarily large number of consecutive mappings. For instance, $[2 \mapsto 3]$ through $F_1$, after which $[3 \mapsto 1]$ through $F_4$. This permutation of mappings through restrictions of $F$, first $F_1$ and then $F_4$, occurs at but not before $[2+2^{1+4}=34]$, and indeed $[2+32\0]$ are those elements of $[\N]$ that are mapped through first $F_1$ and then $F_4$. I define this property as $z$-proportionality:
\begin{definition}
\label{Z-prop}
$z$-proportionality: Let the successive restrictions $F_z$ through which some $[x \in \N]$ maps be indexed $i=1,2,3,\ldots$, so that some $[x \in \N]$ is mapped, successively, through $F_{z_1},F_{z_2},\ldots,F_{z_n}$ (in that order). Then this same permutation of mappings through $F_{z_1},F_{z_2},\ldots,F_{z_n}$ occurs at but not before $[x+2^{\sum_{i=1}^n z_i}]$. Indeed, of any and all $2^{\sum_{i=1}^n z_i}$ consecutive elements of $[\N]$, exactly one is mapped through, successively, $F_{z_1}, F_{z_2},\ldots,F_{z_n}$.
\end{definition}
Let sets that have this property, including $[\N]$, be called $z$-proportional. To see why $[\N]$ is $z$-proportional, consider the subset of $[\N]$ that is mapped through some $F_z$. This subset of $[\N]$ is mapped through $F_z$ to either $[1+3\0]$ or $[3+3\0]$. The interval at which elements of these subsets are found in $[\N]$ is 3, which is co-prime with the interval at which elements mapped through any restriction $F_z$ are found in $[\N]$, these intervals being $2^z$. Therefore also in $[1+3\0]$ or $[3+3\0]$, it is true that of any and all $2^z$ consecutive elements, exactly one is mapped through $F_z$ (conform with Lemma 1, Appendix A1). For any restriction $F_z$, these elements are found in $[\N]$ at intervals of $3 \cdot 2^z$. The mapping through $F_z$ then divides $2^z$ out of the interval and multiplies by 3, yielding an interval of 9, co-prime with the intervals $2^z$ at which elements mapped through any restriction $F_z$ are found in $[\N]$, and so on. 

If the Collatz conjecture is true, then $F$ organizes $[\N]$ in a tree rooted in [1]. A natural way to proceed is then to start with [1], and verify which elements of $[\N]$ map to it. Since $[1 \mapsto 1]$, $[E^{\0}(1) \mapsto 1]$. Next, we could verify which elements map to $[E^{\0}(1)]$. For instance, since $[E^1(1)=3]$, and $[2 \mapsto 3]$, we find that $[E^{\0}(2) \mapsto 3]$. We could then ask which elements map to $[E^{\0}(2)]$, and so forth.

This algorithm suggests splitting up $F$ in two parts: $F_l: [2+2\0 \cup 1+4\0 \rightarrow 3+3\0 \cup 1+3\0]$, such that
\begin{equation}
\label{Fl}
F_l([2+2m]):=[3+3m]|m \in \0, \; F_l([1+4m]):=[1+3m]|m \in \0;
\end{equation}
the part of $F$ that does not involve mappings of elements that have a lower equivalent, which is one-to-one, versus the part of $F$ that involves taking equivalents. Thus, $F_l=F|_{[2+2\0 \cup 1+4\0]}$ consists of $F_1=F|_{[2+2\0]}$ and $F_2=F|_{[1+4\0]}$ (it could be said that $l=1,2$, although the subscript $l$ is really just an indication to refer to the part of $F$ that maps elements of $[\N]$ that have no lower equivalent; $l$ stands for `lower part'). The other restrictions of $F$, $F_z, \; z>2$, have a domain that is derived from the domain of either $F_1$ or $F_2$ through application of $E$. Indeed, $[E^{\N}(2+2\0 \cup 1+4\0)=3+4\0]$.

To verify the elements that are included in the tree rooted in [1] (or in any other root, for that matter), start with the root, take the equivalents of the elements in the root, then take the elements of $[\N]$ that map to those equivalents through $F_l$, then take the equivalents of those elements, and so forth. The result of this algorithm seems to partition $[\N]$ in trivial loop [1] and `strings', ordered subsets of $[\N]$ that run from an element of $[2+3\0]$ to an element of $[3+4\0]$ through mapping $F_l$ only. An example of a string is $[5,4,6,9,7]$: $[5 \in 2+3\0]$, $[7 \in 3+4\0]$, with [5], [4], [6], and [9] in the domain of $F_l$. To see how this partition arises, notice that taking equivalents of $[x]$ means finding those elements of $[3+4\0]$ that have the same image as $[x]$. The elements that map to those equivalents are those that map there through $F_l$, and the equivalents of these elements. Thus the question follows whether recursive application of $F_l$ on $[2+3\0]$, which is not in the range of $F_l$, always leads to one of these equivalents (Figure 1).

As the algorithm above demonstrates, strings and strings only are included in trees. Hence, any element of $[\N]$ that is not in a string, is in a root. Importantly, a hypothetical finding that not all elements of $[\N]$ are in strings would therefore disprove the conjecture. Reversely, a finding that under $F$, $[\N]$ is partitioned in the trivial loop $[1]$ and strings would be in accordance with the conjecture. Moreover, the string partition means that all trajectories go through at least one element of $[3+4\0]$ (for odd natural numbers this would be $5+8\0$), which is interesting in its own right.

\begin{figure}[h]
     \centering
     \includegraphics[width=\textwidth]{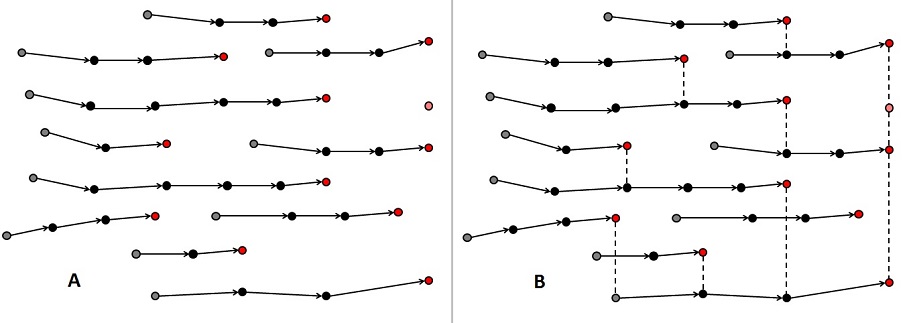}
     \caption{\small{\textbf{A.} $[\N \setminus 1]$ is partitioned in strings. Elements of $[3+4\0]$, which have a first lower equivalent, are colored red. Elements of $[2+3\0]$, that are not in the domain of $F_l^{-1}$, are partially transparent. Internal vertices are just black. The cardinality of the strings differs. One string has only one element, which is true for $[\{2+3\0\}\cap\{3+4\0\}]$, and is therefore transparent red. Other strings have various cardinality. The representation is schematic: in reality, strings cross, but do not intersect. For instance, the strings $\{[8,12,18,27]\}$ and $\{[17,13,10,15]\}$ (both not necessarily depicted here) cross, since the former begins lower and ends higher than all elements of the latter. Yet, they do not have elements in common: they do not intersect. \newline \textbf{B.} The dashed lines relate nodes to their first higher equivalent, if it is depicted. Notice how $[3+4\0]$ may be the first higher equivalent of any element of another string. Notice also how strings act as collectors for other strings: all except one of the strings depicted here, composed of a total of 39 nodes, drain into the lower string, which in turn will be collected by yet another string or by the trivial loop. Finally, notice how the bottom right node, itself in $[3+4\0]$, has its first, second, and third higher equivalent connected to it. In fact, every node has infinitely many higher equivalents.}}
\end{figure}

\section{Why we might think that there is a string partition}
The goal is to show that $[\N \setminus 1]$ is partitioned in strings. Since $[2+3\0]$ is not in the range of $F_l$, it should be shown that recursive application of $F_l$ on $[2+3\0]$ all of $[\N \setminus 1 \setminus 2+3\0=3+3\0 \cup 4+3\0]$ is hit: first $F_l$ is applied on $[2+3\0]$, then on the image of $[2+3\0]$ when mapped through $F_l$, then on the image of the image of $[2+3\0]$ and so forth indefinitely. Notice that:
\begin{enumerate}
	\item When an element of $[3+4\0]$ is hit, a string ends, as $[3+4\0]$ is not in the domain of $F_l$: these elements generate no image in the next iteration.
	\item All elements hit through this procedure are indeed in $[3+3\0 \cup 4+3\0]$, since this is the range of $F_l$ if $[1 \mapsto 1]$ is ignored.
	\item All elements of $[\N]$ that are hit through this procedure, are hit exactly once, since $F_l$ is one-to-one and $[2+3\0]$ (the starting point) is not in the range of $F_l$.
\end{enumerate}
It remains to be shown that \emph{all} of $[3+3\0 \cup 4+3\0]$ is indeed hit. 

Similarly, the inverse of $F_l$, $F_l^{-1}: [3+3\0 \cup 1+3\0 \rightarrow 2+2\0 \cup 1+4\0]$, such that
\begin{equation}
\label{FlInv}
F_l^{-1}([3+3m])=[2+2m]|m \in \0, \; F_l^{-1}([1+3m])=[1+4m]|m \in \0,
\end{equation}
could be applied recursively on $[3+4\0]$: first $F_l^{-1}$ is applied on $[3+4\0]$, then on the image of $[3+4\0]$ when mapped through $F_l^{-1}$, then on the image of the image of $[3+4\0]$ and so forth indefinitely. Does this include all of $[\N \setminus 1 \setminus 3+4\0]$ in strings? In this case,
\begin{enumerate}
	\item When an element of $[2+3\0]$ is hit, a string ends, as $[2+3\0]$ is not in the domain of $F_l^{-1}$: these elements generate no image in the next iteration.
	\item All elements hit through this procedure are indeed in $[2+2\0 \cup 5+4\0 = \N \setminus 1 \setminus 3+4\0]$, since this is the range of $F_l^{-1}$ if $[1 \mapsto 1]$ is ignored.
	\item All elements of $[\N]$ that are hit through this procedure, are hit exactly once, since $F_l^{-1}$ is one-to-one and $[3+4\0]$ (the starting point) is not in the range of $F_l^{-1}$.
\end{enumerate}
It remains to be shown that \emph{all} of $[2+2\0 \cup 5+4\0]$ is indeed hit.

If it could be shown that recursive application of $F_l$ on $[2+3\0]$ hits all of $[3+3\0 \cup 4+3\0]$, which in union with $[2+3\0]$ equals $[\N \setminus 1]$, while it could also be shown that recursive application of $F_l^{-1}$ on $[3+4\0]$ hits all of $[2+2\0 \cup 5+4\0]$, which in union with $[3+4\0]$ equals $[\N \setminus 1]$ as well, then the conclusion would inevitably arise that all the ends meet, and that $[\N \setminus 1]$ is partitioned in strings that run from an element of $[2+3\0]$ to an element of $[3+4\0]$. This means that all trajectories except the trivial loop go through $[3+4\0]$.

It is necessary to apply the pigeonhole principle: ``of any and all $\mathcal{N}$ consecutive elements of $[\N \setminus 1]$, exactly $\mathcal{N}$ are included in the strings". This pigeonhole principle can be applied because the subsets of $[\N \setminus 1]$ that show up in the analysis are periodic in the \emph{ensemble direction}: they are of the form $[a+b\0]$, with $a \in [\N]$ and $b \in \N$. Provided that $[a] \leq b$, we can make the statement that ``of any and all $b$ consecutive elements of $[\N]$, exactly one is in $[a+b\0]$". Let $a$ in this expression be called the `intercept', while $b$ is, as before, the `interval'. The intercept requires special attention, because if $[a] > b$, the statement that ``of any and all $b$ consecutive elements of $[\N]$, exactly one is in $[a+b\0]$" would not be true: in the first $a-1 \geq b$, none would be in $[a+b\0]$.

\subsection{Does recursive application of $F_l$ on $[2+3\0]$ yield that all of $[\N \setminus 1]$ is included in the strings?}
Let $[A_0:=2+3\0]$, $[A_{1}:=F_l(2+3\0)]$ and so forth, so that $[A_{k+1}:=F_l(A_k)]$. All $[A_k]$ are periodic in the ensemble. For instance, $[A_1=3+9\0 \cup 4+9\0]$ consists of two periodic subsets, $[3+9\0]$ and $[4+9\0]$, that have intercepts [3], respectively [4], and interval 9.

The interval in $[\N]$ between any two consecutive elements that are mapped through $F_1$ is 2, while for $F_2$ this is 4, while for those elements mapped through $F_{>2}$ (i.e. $[3+4\0]$) this is 4 as well. As the interval of $[2+3\0]$ is co-prime with 2 and 4, it follows that this is true also for $[2+3\0]$ (because of Lemma 1 in Appendix A1). As a result, of the first four elements of $[2+3\0]$, indeed for any and all four consecutive elements of $[2+3\0]$, and indeed for any and all four consecutive elements in a $z$-proportional set, exactly two (with an interval of 2) are mapped through $F_1$, exactly one is mapped through $F_2$, and exactly one is mapped through $F_{>2}$, i.e. is in $[3+4\0]$ and does not generate an image in the next iteration. Starting from $[2+3\0]$, the subsets of $[\N \setminus 1]$ that are obtained through repeated application of $F_l$ all are $z$-proportional (conform with Lemma 1 in Appendix A1), while the number of $z$-proportional subsets doubles every mapping: Those elements that are mapped through $F_1$, respectively $F_2$, are found at intervals in $[\N]$ that are multiples of 2, respectively 4. $F_l$ divides this interval by 2, respectively 4, and multiplies it by 3. Hence, the intervals of subsets of $[A_k]$ are $3^{k+1}$, co-prime with 2 and 4, so that $z$-proportionality is maintained over an arbitrarily large number of recursive applications of $F_l$. Every application of $F_l$, two new $z$-proportional subsets are formed for each $z$-proportional subset in the previous iteration: one through $F_1$ and one through $F_2$. Hence, $[A_k]$ consist of $2^k$ $z$-proportional subsets with interval $3^{k+1}$ in $[\N]$.

As the periodic subsets of $A_k$ have intervals $3^{k+1}$, we can divide up $[\N \setminus 1]$ in sections of $3^m$ consecutive elements and ask how many of these are in $\bigcup A_k, \; 0 \leq k<m$. Start with $m=1$, then $m=2$, then $m=3$ and so forth, finally letting $m \rightarrow \infty$.

Of any and all three consecutive elements of $[\N \setminus 1]$, exactly one is in $[A_0=2+3\0]$. Of any and all nine consecutive elements of $[\N \setminus 1]$, exactly two are in $[A_1=3+9\0 \cup 4+9\0]$. Also, exactly three are in $[A_0]$. $[A_2=18+27\0 \cup 16+27\0 \cup 6+27\0 \cup 10+27\0]$, so that of any and all 27 consecutive elements of $[\N \setminus 1]$, exactly four are in $[A_2]$. In addition, of any and all 27 consecutive elements of $[\N \setminus 1]$, exactly six are in $[A_1]$, while exactly 9 are in $[A_0]$, and therefore that exactly nineteen are in $[A_0 \cup A_1 \cup A_2]$. And so forth. We hence find that of $3^m$ consecutive elements of $[\N \setminus 1]$, exactly
\begin{equation}
\sum_{k=0}^m 2^k \cdot 3^{m-k-1}
\end{equation}
are included in $[\bigcup A_k, \; 0 \leq k < m]$. Since
\begin{equation}
\lim_{m \rightarrow \infty} \sum_{k=0}^m 2^k/3^{k+1} = 1,
\end{equation}
it follows that
\begin{equation}
\label{limit1}
\lim_{m \rightarrow \infty} \sum_{k=0}^m 2^k \cdot 3^{m-k-1}=3^m,
\end{equation}
the desired result.

However, care should be taken of the intercepts, i.e., for the first $3^m$ elements of $[\N \setminus 1]$: from which point onwards can we take sections of $[\N \setminus 1]$ of $3^m$ consecutive elements that all look identical, so that we can be sure that the above is true? Perhaps it could be possible only to assure consecutive sections of $[\N \setminus 1]$ of $3^m$ consecutive elements of $[\N \setminus 1]$ starting from, say, [10], so that $[[10,10+3^k),[10+3^k,10+2 \cdot 3^k),[10+2 \cdot 3^k,10+3 \cdot 3^k),\ldots]$ all look the same in terms of periodicity. Then it remains to be seen if [3], [4], [6], [7] and [9] are indeed included in strings. Therefore the following Lemma.
\begin{lemma}
\label{Ck}
Let $[C_k]$ and $[V_k]$ be the intercept respectively the interval of some $z$-proportional subset of $[A_k]$. Then it holds that $[C_k<V_k]$ for all $k$.
\end{lemma}
\begin{proof}
By induction. In the first four elements of a $z$-proportional set, all of $F_1$, $F_2$ and $F_{>2}$ occur and the ordering of how $F_1$, $F_2$ and $F_{>2}$ occur throughout the entire $z$-proportional subset is fixed. The possible orderings are $F_1, F_2/F_{>2}, F_1, F_{>2}/F_2$ and $F_2/F_{>2}, F_1, F_{>2}/F_2, F_1$. Once the ordering is known for the first four elements of a $z$-proportional set, it is known through the entire $z$-proportional set. $[A_k]$ is made up of $z$-proportional subsets of $[\N]$ for all $k$.

Now take some $[A_k]$ and apply $F_l$ to obtain $[A_{k+1}]$. The intercepts of the newly formed $z$-proportional subsets that make up $[A_{k+1}]$ are formed through application of $F_l$ on the first four elements of each of the $z$-proportional subsets that make up $[A_k]$: one intercept derives from application of $F_1$, the other from application of $F_2$. These first four elements are: $[C_k, C_k+V_k, C_k+2V_k, C_k+3V_k]$. The highest possible new intercept $[C_{k+1}]$ is formed if $[C_k+3V_k]$ is mapped through $F_2$ (it is readily verified that all other possible intercepts are lower). Thus,
\begin{equation}
[C_{k+1} \leq 3(C_k+3V_k-1)/4+1].
\label{Ck+1}
\end{equation}
Provided that $[C_k \leq V_k]$ (i.e., assuming this Lemma is true), this gives
\begin{equation}
[C_{k+1} < 3 \cdot 4V_k/4=V_{k+1}].
\label{Ck+12}
\end{equation}
Substituting $k$ for $k+1$ in equation (\ref{Ck+12}) and noticing that for $[A_1=3+9\0 \cup 4+9\0]$, the intercepts of its $z$-proportional subsets are $[3]<9$ and $[4]<9$, conform with the lemma statement, completes the proof.
\end{proof}

The intercepts of $[A_1=3+9\0 \cup 4+9\0]$, $[3]$ and $[4]$, are the lowest elements of $[3+3\0]$ and $[4+3\0]$, the part of $[\N \setminus 1]$ that is not in $[2+3\0]$. Together with Lemma (\ref{Ck}), which assures that the intercepts remain below the interval for all $z$-proportional subsets that make up all $[A_k]$, we are thus sure that the identical sections of $3^m$ consecutive elements of $[\N \setminus 1]$ start at 2: all the intercepts of $[A_k, \; k>1]$ fall between $[4]$ and $[3+3^{k+1}]$. 

Thus, we have assured that for all the $z$-proportional subsets that make up $[A_k]$ the intercept does not exceed the interval. This means that we can make statements of the sort: ``of any and all $3^m$ consecutive elements of $[\N \setminus 1]$, exactly so many are in the strings'', \emph{including the first $3^m$ elements}. We have also seen that $\lim_{m \rightarrow \infty} \sum_{k=0}^m 2^k \cdot 3^{m-k-1}=3^m$ out of any and all $3^m$ consecutive elements of $[\N \setminus 1]$ are eventually included in the strings. This gives an exact match between the number of pigeons and the number of pigeonholes.

However, this argument relies on taking a limit, and it is unsure whether such a counting process could just by itself constitute proof. Consider the point at which we have $A_0=2+3\0$ and $[A_1=3+9\0 \cup 4+9\0]$. All $[x]\leq27$ will now have to appear in some $A_{k>1}$ as an intercept; since we have no way of predicting intercepts, how can we assure that this happens? From the above it also follows that in any and all $3^m$ elements of $[\N]$, exactly $3^m-\sum_{k=0}^m 2^k \cdot 3^{m-k-1}=2^m$ are not in $[\bigcup A_k, \; 0 \leq k < m]$. What of these positions? It turns out that $2^m$ is exactly the number of pigeonholes necessary for the pigeons that are included in strings at a later point, i.e. that are in $[\bigcup A_k, \; k \geq m]$, as the following shows.

When evaluating $3^m$ consecutive elements of $[\N \setminus 1]$, we find that exactly $\sum_{k=0}^m 2^k \cdot 3^{m-k-1}$ of these are in $[\bigcup A_k, \; 0 \leq k < m]$. All $[A_k], k < m-1$ have already set up the further inclusion of positions in strings: $[A_0]$ has set up $[A_1]$, $[A_1]$ has set up $[A_2]$, and so forth. The only $z$-proportional subsets that are still `live' are $[A_{m-1}]$: these will set up $[A_m]$, which will set up $[A_{m+1}]$, and so forth as the string formation process is continued. Since $[A_{m-1}]$ consists of $2^{m-1}$ $z$-proportional subsets, there are $2^{m-1}$ `live' intercepts. In the next $3^m$ consecutive elements of $[\N \setminus 1]$, there are, again, $2^{m-1}$ `live' positions, and so forth. How do these relate to mappings into each $3^m$ consecutive elements of $[\N \setminus 1]$?

In $[\N \setminus 1]$, indeed in any $z$-proportional subset of $[\N \setminus 1]$, in any and all four consecutive positions, exactly 2 map through $F_1$, exactly one maps through $F_2$, while exactly one does not map through $F_l$. $F_1$ multiplies by 3/2, while $F_2$ multiplies, allowing rounding, by 3/4. Hence, all elements in the first $\mathcal{N} \cdot 2/3$ consecutive elements of $[\N \setminus 1]$ that map through $F_1$ map into the first $\mathcal{N}$ consecutive elements of $[\N \setminus 1]$. These are $\mathcal{N} \cdot 2/3 \cdot 1/2 = \mathcal{N} \cdot 1/3$ positions, each giving one map in the $\mathcal{N}$ positions. Meanwhile, all elements in the first $\mathcal{N} \cdot 4/3$ consecutive elements of $[\N \setminus 1]$ that map through $F_2$ map into the first $\mathcal{N}$ consecutive elements of $[\N \setminus 1]$. These are $\mathcal{N} \cdot 4/3 \cdot 1/4 = \mathcal{N} \cdot 1/3$ positions, each giving one map in the $\mathcal{N}$ positions. Taken together, then, a total of $2/3 \cdot \mathcal{N}$ will be hit if we apply $F_l$ once on $[\N \setminus 1]$, distributed uniformly across the first $\mathcal{N}$ consecutive elements of $[\N \setminus 1]$. The same is true for any and all $\mathcal{N}$ consecutive elements of $[\N \setminus 1]$: fully in line with the range of $F_l$ being $\{1+3\0 \cup 3+3\0\}$, two thirds of any $\mathcal{N}$ consecutive elements of $[\N \setminus 1]$ is hit. For the image of the mapping of $[\N \setminus 1]$ through $F_l$, the same goes, so that $(2/3)(2/3)=4/9$ of any $\mathcal{N}$ consecutive elements of $[\N \setminus 1]$ is hit. And so forth, so that the geometric series $2/3+4/9+8/27+\ldots=2$ is formed as $\mathcal{N} \rightarrow \infty$. Without loss of generality, we can let $\mathcal{N} \rightarrow \infty$ while letting $[\mathcal{N} \in 3+4\0]$, so that the rounding for $F_2$ does not upset the result. On the average, it is also exactly true for $\mathcal{N}$ not going to infinity: for each live position within $3^m$ consecutive elements of $[\N \setminus 1]$, exactly 2 positions are hit within those $3^m$ consecutive elements of $[\N \setminus 1]$ (which does not mean that these 2 positions are also hit by one of the live positions within the same $3^m$ consecutive elements of $[\N \setminus 1]$; indeed, often they are not).

All $z$-proportional subsets of $[A_{m-1}]$ exhibit the exact same mapping behavior as $[\N \setminus 1]$; only the intercept may differ. Thus, as $\mathcal{N} \rightarrow \infty$, $[A_{m-1}]$ gives a mapping in $[\N \setminus 1]$ that is the same as the mapping of $[\N \setminus 1]$, except for its density being reduced by a factor $2^{m-1}/3^m$, such that it gives exactly $2\cdot2^{m-1}=2^m$ hits within $3^m$ consecutive elements of $[\N  \setminus 1]$ as $\mathcal{N} \rightarrow \infty$. On the average, it is also exactly true for $\mathcal{N}$ not going to infinity, say $3^m$. There will then be significant spillover between bins of $3^m$ consecutive elements of $[\N \setminus 1]$: some of these bins will give more than their share of pigeons, others less. This average takes account of locality, i.e. pigeons and pigeonholes. Combining this with the fact that the intercepts of all $A_k, \; k \geq m$ will continue to not exceed the intervals at all times, we conclude that the match between the number of pigeons and the number of pigeonholes is indeed exact.

Hence, it is \emph{necessary} that within any and all $3^m$ consecutive elements of $[\N \setminus 1]$, $2^m$ are not in $[\bigcup A_k, \; k < m]$: these positions serve as pigeonholes for the pigeons that emerge as the string formation process progresses. If they were not open, this would mean that some positions at a later stage in the string formation process have nowhere to map, while we do know that they map somewhere: a contradiction. However, there is an exact match between the number of live positions and the number of open pigeonholes within each section of $3^m$ consecutive elements of $[\N \setminus 1]$.

Thus, we have good reasons to believe that recursive application of $F_l$ on $[2+3\0]$ includes all of $[\N \setminus 1]$ in strings.

\subsubsection{Two potential pitfalls in the interpretation of the argument}
First, it is helpful to point out that the proof method above does not lead to the same conclusion for $3n-1$ numbers, where $3n+1$ in equation (\ref{Collatz}) is replaced by $3n-1$. In this case, strings also start at $[2+3\0]$ and end at $[3+4\0]$, but $[2+3\0]$ map to $[6+9\0 \cup 7+9\0]$, leaving positions $[3]$ and $[4]$ out of the intercept. Verifying where these positions map manually, as we did for recursive application of $F_l^{-1}$ on $[3+4\0]$ above, it turns out that these positions are elements of no strings, instead forming a cycle. This is not the case for $3n+1$ numbers (see above), which demonstrates the rationale of verifying the intercepts.

Second, in the argument above it was stated that for every live position in $3^m$ consecutive elements of $[\N \setminus 1]$, 2 positions are being hit at a later point within those $3^m$ consecutive elements of $[\N \setminus]$. This does not mean that those 2 positions are indeed hit from within those $3^m$ consecutive elements of $[\N \setminus 1]$. In the same vein, the argument is not that the collection of intercepts be $z$-proportional. Indeed, the collection of intercepts for any $A^k$ may be far from $z$-proportional; the intercepts of $A_2$, for example, are $[6,10,16,18]$, which all map through $F_1$. As discussed, there will be significant overflow between bins of $3^m$ consecutive elements of $[\N \setminus 1]$. In any and all $3^m$ consecutive elements of $[\N \setminus 1]$, there are exactly $2^m$ pigeonholes, and exactly $2^{m-1}$ live positions. On the average, for each live position within $3^m$ consecutive elements of $[\N \setminus 1]$, exactly 2 positions are hit within those $3^m$ consecutive elements of $[\N \setminus 1]$.

\subsection{Does recursive application of $F_l^{-1}$ on $[3+4\0]$ yield that all of $[\N \setminus 1]$ is included in the strings?}
I now evaluate recursive application of $F_l^{-1}$ on $[3+4\0]$, which is essentially a reverse copy of the above. Recall that $F_l^{-1}: [3+3\0 \cup 1+3\0 \rightarrow 2+2\0 \cup 1+4\0]$, such that 
\begin{equation}
F_l^{-1}([3+3m])=[2+2m]|m \in \0, \; F_l^{-1}([1+3m])=[1+4m]|m \in \0.
\end{equation}
Recall that $F_l^{-1}$ consists of $F_1^{-1}$ and $F_2^{-1}$. Whether an element of $[\N]$ is mapped through $F_1^{-1}$, $F_2^{-1}$, or is outside of the domain of $F_l^{-1}$ depends on its residual (mod 3). Let this residual be denoted $y:[\N] \rightarrow \{0,1,2\}$:
\begin{equation}
y([x]):= k \in \{0,1,2\}: ([x]-k)/3 \in \N.
\label{y}
\end{equation}
If $y([x])=0$, $[x]$ is mapped to a smaller element of $[\N]$ through $F_1^{-1}$, if $y([x])=1$, $[x]$ is mapped to a larger element of $[\N]$ through $F_2^{-1}$, while if $y([x])=2$, $[x]$ is not in the domain of $F_l^{-1}$. 

Clearly, of any and all three consecutive elements of $[\N]$, exactly one is mapped through $F_1^{-1}$, exactly one is mapped through $F_2^{-1}$, and exactly one is not mapped through $F_l^{-1}$. Those element of $[\N]$ that are mapped through either $F_1^{-1}$ or $F_2^{-1}$ are found at intervals of 3. $F_1^{-1}$ or $F_2^{-1}$ divides 3 out of the interval and multiplies by, respectively, 2 or 4 to produce a new interval. This is co-prime with 3, so that because of Lemma 1 in Appendix A1, also in any and all three consecutive elements of this new subset of $[\N]$, exactly one is mapped through $F_1^{-1}$, exactly one is mapped through $F_2^{-1}$, and exactly one is not mapped through $F_l^{-1}$. And so on. I define this property:
\begin{definition}
\label{Y-prop}
$y$-proportionality: If $[x \in \N]$ is mapped through some permutation of $n$ occurrences of $F_1^{-1}$ or $F_2^{-1}$, then this same permutation of mappings occurs at but not before $[x+3^n]$. Indeed, of any and all $3^n$ consecutive elements of $[\N]$, exactly one is mapped through a specified permutation of $n$ occurrences of $F_1^{-1}$ or $F_2^{-1}$ (in that order).
\end{definition}
For instance, $[7]$ maps through $F_1^{-1}$ to $[9]$, $[9]$ maps through $F_2^{-1}$ to $[6]$, $[6]$ maps through $F_2^{-1}$ to $[4]$, and $[4]$ maps through $F_1^{-1}$ to $[5]$. This exact same permutation of $F_1^{-1},F_2^{-1},F_2^{-1},F_1^{-1}$, in that order, occurs at but not before $[7+3^4=88]$.
 
Sets that have this property I call `$y$-proportional'. Following Lemma (1 of Appendix A1), any subset of $[\N]$ taken from $[\N]$ with a period that is co-prime with 3 is $y$-proportional. Hence, the ranges of $F_1^{-1}$ and $F_2^{-1}$ are $y$-proportional, as is $[3+4\0]$.

Let $[B_0=3+4\0]$, let $[B_1:=F_l^{-1}([3+4\0)]$ and let $[B_{k+1}:=F_l^{-1}(B_k)]$. Then $[B_1=2+8\0 \cup 9+16\0]$. Next, $[B_2=F_l^{-1}(B_1)=12+16\0 \cup 13+32\0 \cup 6+32\0 \cup 33+64\0]$ and so forth. Every iteration, application of $F_l^{-1}$ on a $y$-proportional subset yields two new $y$-proportional subsets; one through $F_1^{-1}$ and one through $F_2^{-1}$. The new $y$-proportional subset that is formed through $F_1^{-1}$ has an interval that is twice the interval of the $y$-proportional subset it was derived from, while the new $y$-proportional subset that is formed through $F_2^{-1}$ has an interval that is four times the interval of the $y$-proportional subset it was derived from. And so on. Every time $F_1^{-1}$ or $F_2^{-1}$ divides out 3 from the interval and multiplies by either 2 or 4 to form the new interval, so that the new interval too is $y$-proportional.

Divide $[\N \setminus 1]$ in sections of $4^m$ consecutive elements of $[\N \setminus 1]$ and evaluate the number of elements in each of those sections that is in $[B_k, \; 0 \leq k < m]$. Start with $m=1$, then $m=2$, and so forth, finally letting $m \rightarrow \infty$. For $m=1$ we find that of any and all four consecutive elements of $[\N \setminus 1]$, exactly one is in $[B_0=3+4\0]$. For $m=2$ we find that of any and all sixteen consecutive elements of $[\N \setminus 1]$, exactly four are in $[B_0]$. Since $[B_1=2+8\0 \cup 9+16\0]$, two are in $[2+8\0]$ while one is in $[9+16\0]$. Thus, in total we find that of any and all sixteen consecutive elements of $[\N \setminus 1]$, exactly seven are in $[B_0 \cup B_1]$. Next, for $m=3$ we find that of any and all 64 consecutive elements of $[\N \setminus 1]$, exactly 28 are in $[B_0 \cup B_1]$. Since $[B_2=F_l^{-1}(B_1)=12+16\0 \cup 13+32\0 \cup 6+32\0 \cup 33+64\0]$, we find that four are in $[12+16\0]$, two are in $[13+32\0]$, two are in $[6+32\0]$, while one is in $[33+64\0]$. This gives a grand total of $28+4+2+2+1=37$ out of any and all 64 consecutive elements of $[\N \setminus 1]$. And so forth. Hence, we find that of $4^m$ consecutive elements of $[\N \setminus 1]$, exactly
\begin{equation}
\sum_{k=0}^m 3^k \cdot 4^{m-k-1}
\end{equation}
are included in $[\bigcup B_k, \; 0 \leq k < m]$. Since
\begin{equation}
\lim_{m \rightarrow \infty} \sum_{k=0}^m 3^k/4^{k+1} = 1,
\end{equation}
it follows that
\begin{equation}
\lim_{m \rightarrow \infty} \sum_{k=0}^m 3^k \cdot 4^{m-k-1}=4^m,
\end{equation}
the desired result. As above, the intercepts should be verified. Hence the following Lemma.
\begin{lemma}
\label{Dk}
Let $[D_k]$ and $[W_k]$ be the intercept respectively the interval of some $z$-proportional subset of $[A_k]$. Then it holds that $[D_k<W_k]$.
\end{lemma}
\begin{proof}
By induction. The highest possible intercept of a new $y$-proportional subset is formed if the third element of the $y$-proportional subset it is derived from maps through $F_2^{-1}$. Thus,
\begin{equation}
[D_{k+1} \leq 4(D_k+2W_k-1)/3+1].
\label{D}
\end{equation}
Assuming that $[D_k < W_k]$, this yields
\begin{equation}
[D_{k+1} < 4(3W_k)/3=4W_k=W_{k+1}].
\label{W}
\end{equation}
Substituting $k$ for $k+1$ in equation (\ref{W}) and noticing that for $[B_1=2+8\0 \cup 9+16\0]$, the intercepts of its $y$-proportional subsets are $[2]<8$ and $[9]<16$, conform with the lemma statement, then completes the proof.
\end{proof}

The intercepts of $[B_1=2+8\0 \cup 9+16\0]$ are \emph{not} the lowest elements of $[\N \setminus 1 \setminus 3+4\0]$, $[2]$ and $[4]$. Thus, we verify manually that $[4]$, $[5]$, $[6]$, and $[8]$ are in strings, which is the case. Together with Lemma (\ref{Dk}), which assures that the intercepts remain below the interval for all $y$-proportional subsets that make up all $[B_k]$, we are thus sure for some large enough $m$ that the identical sections of $4^m$ consecutive elements of $[\N \setminus 1]$ start at 2: all the intercepts of $[B_k, \; k>1]$ fall between $[4]$ and $[2+4^{k+1}$].

Thus, we have assured that for all the $y$-proportional subsets that make up $[B_k]$ the intercept does not exceed the interval. This means that we can make statements of the sort: ``of any and all $4^m$ consecutive elements of $[\N \setminus 1]$, exactly so many are in the strings'', including the first $4^m$ elements. We have also seen that $\lim_{m \rightarrow \infty} \sum_{k=0}^m 3^k \cdot 4^{m-k-1}=4^m$ of out of any and all $4^m$ consecutive elements of $[\N \setminus 1]$ are eventually included in the strings.

As above, are we sure that such a counting process is sufficient? For instance, if we have $m=3$ and consider $\bigcup B_k, ;\ k=0,1,2$, any $[x]\leq 256$ can be included only as an intercept of some $y$-proportional subset; since we have no way of predicting the intercepts, how can we be sure that this happens? From the above it also follows that in any and all $4^m$ elements of $[\N \setminus 1]$, exactly $4^m-\sum_{k=0}^m 3^k \cdot 4^{m-k-1}=3^m$ are not in $[\bigcup B_k, \; 0 \leq k < m]$. What of these positions? Again, it turns out that $3^m$ is exactly the number of pigeonholes necessary for the pigeons that are included in strings at a later point, i.e. that are in $[\bigcup B_k, \; k \geq m]$, as the following shows. 

When evaluating $4^m$ consecutive elements of $[\N \setminus 1]$, we find that exactly $\sum_{k=0}^m 3^k \cdot 4^{m-k-1}$ of these are in $[\bigcup B_k, \; 0 \leq k < m]$. All $[B_k], k < m-1$ have already set up the further inclusion of positions in strings: $[B_0]$ has set up $[B_1]$, $[B_1]$ has set up $[B_2]$, and so forth. The only $y$-proportional subsets that are still `live' are $[B_{m-1}]$: these will set up $[B_m]$, which will set up $[B_{m+1}]$, and so forth as the string formation process is continued. $[B_{m-1}]$ consists of $2^{m-1}$ $y$-proportional subsets, with intervals such that there are $3^{m-1}$ `live' positions in any and all $4^m$ consecutive elements of $[\N \setminus 1]$. How do these relate to mappings into each $4^m$ consecutive elements of $[\N \setminus 1]$?

In $[\N \setminus 1]$, indeed in any $y$-proportional subset of $[\N \setminus 1]$, of any and all three consecutive elements, exactly one maps through $F_1^{-1}$, exactly one maps through $F_2^{-1}$ and exactly one does not map through $F_l^{-1}$. $F_1^{-1}$ multiplies by 2/3, while $F_2^{-1}$ multiplies, allowing rounding, by 4/3. Hence, all elements in the first $\mathcal{N} \cdot 3/2$ consecutive elements of $[\N \setminus 1]$ that map through $F_1^{-1}$ map into the first $\mathcal{N}$ consecutive elements of $[\N \setminus 1]$. These are $\mathcal{N} \cdot 3/2 \cdot 1/3 = \mathcal{N} \cdot 1/2$ positions, each giving one map in the $\mathcal{N}$ positions. Meanwhile, all elements in the first $\mathcal{N} \cdot 3/4$ consecutive elements of $[\N \setminus 1]$ that map through $F_2^{-1}$ map into the first $\mathcal{N}$ consecutive elements of $[\N \setminus 1]$. These are $\mathcal{N} \cdot 3/4 \cdot 1/3 = \mathcal{N} \cdot 1/4$ positions, each giving one map in the $\mathcal{N}$ positions. Taken together, then, a total of $3/4 \cdot \mathcal{N}$ will be hit if we apply $F_l$ once on $[\N \setminus 1]$. The same is true for any and all $\mathcal{N}$ consecutive elements of $[\N \setminus 1]$: fully in line with the range of $F_l^{-1}$ being $\{2+2\0 \cup 5+4\0\}$, three quarters of any and all $\mathcal{N}$ consecutive elements of $[\N \setminus 1]$ is hit. For the image of the mapping of $[\N \setminus 1]$ through $F_l^{-1}$, the same goes, so that $(3/4)(3/4)=9/16$ of any $\mathcal{N}$ consecutive elements of $[\N \setminus 1]$ is hit. And so forth, so that the geometric series $3/4+9/16+27/64+\ldots=3$ is formed as $\mathcal{N} \rightarrow \infty$. Without loss of generality, we can let $\mathcal{N} \rightarrow \infty$ while letting $[\mathcal{N}\N \in 2+3\0]$, so that the rounding for $F_2$ does not upset the result and this is exactly true for all $\mathcal{N}$ consecutive positions. On the average, it is also exactly true for $\mathcal{N}$ not going to infinity: for each live position within $4^m$ consecutive elements of $[\N \setminus 1]$, exactly 3 positions are hit within those $4^m$ consecutive elements of $[\N \setminus 1]$ (which does not mean that these 3 positions are also hit by one of the live positions within the same $4^m$ consecutive elements of $[\N \setminus 1]$; indeed, often they are not).

All $y$-proportional subsets of $[B_{m-1}]$ have the exact same mapping behavior as $[\N \setminus 1]$; only the intercept may differ. Thus, as $\mathcal{N} \rightarrow \infty$, $[B_{m-1}]$ gives a mapping in $[\N \setminus 1]$ that is the same as the mapping of $[\N \setminus 1]$, except for its density being reduced by a factor $3^{m-1}/4^m$, such that it gives exactly $3\cdot3^{m-1}=3^m$ hits within $4^m$ consecutive elements of $[\N  \setminus 1]$ as $\mathcal{N} \rightarrow \infty$. On the average, it is also exactly true for $\mathcal{N}$ not going to infinity, say $4^m$. There will then be significant spillover between bins of $4^m$ consecutive elements of $[\N \setminus 1]$: some of these bins will give more than their share of pigeons, others less. This average takes account of locality, i.e. pigeons and pigeonholes. Combining this with the fact that the intercepts of all $B_k, \; k \geq m$ will continue to not exceed the intervals at all times, we conclude that the match between the number of pigeons and the number of pigeonholes is indeed exact.

Hence, it is \emph{necessary} that within any and all $4^m$ consecutive elements of $[\N \setminus 1]$, $3^m$ are not in $[\bigcup B_k, \; k < m]$: these positions serve as pigeonholes for the pigeons that emerge as the string formation process progresses. If they were not open, this would mean that some positions at a later stage in the string formation process have nowhere to map, while we do know that they map somewhere: a contradiction. However, there is an exact match between the number of live positions and the number of open pigeonholes within each section of $4^m$ consecutive elements of $[\N \setminus 1]$.

Thus, we have good reasons to believe that recursive application of $F_l^{-1}$ on $[3+4\0]$ yields that all of $[\N \setminus 1]$ is included in the strings.

\subsection{Ends meet}
The finding that all elements of $[\N \setminus 1]$ are reached by recursive application of $F_l$ on $[2+3\0]$ (section 4.1), while all elements of $[\N \setminus 1]$ are reached by recursive application of $F_l^{-1}$ on $[3+4\0]$ (section 4.2), means that all of $[\N \setminus 1]$ is reached from either end. Ends must then meet, and it follows that under the conjugate Collatz map $F$, $[\N \setminus 1]$ is partitioned in strings that run from an element of $[2+3\0]$ to an element of $[3+4\0]$.

\section{Why we might think that this partition is interesting}
The string partition, if true, means that although there is no limit to the cardinality of a string (indeed, take any natural number $n \in \N$ and there are infinitely many strings that have this cardinality, spaced according to the law of $z$-proportionality (or $y$-proportionality in the opposite direction)), all strings start with an element of $[2+3\0]$ and end with an element of $[3+4\0]$: it may take infinitely long, but every single mapping through $F_l$ meets an element of $[3+4\0]$. I have succesfully tested this in a simulation up to element $[159902416]$, where R loses precision in the modulus due to floats. A manual test with help of https://www.dcode.fr/collatz-conjecture confirmed that the trajectory of $[159902416]$ too ends up in an element of $[3+4\0]$. Translating this back to natural numbers, this means that no trajectory in the Collatz mapping can avoid going through $5+8\0$.

Consider the generalization
\begin{equation}
\label{Generalization}
	C_p(n)= 	\begin{cases}
					n/2 &\mbox{if n is even}; \\ 
					3n+p & \mbox{if n is odd}, 
			\end{cases} 
\end{equation}
\newline
with $p \in \{\ldots,-3,-1,1,3,\ldots\}$ ($p$s that are even do not make sense). For every member of this generalization, there are trivial loops $\{p\}$ and $\{-p\}$, and there is an equivalence function $[E_p(x):=4x+q]$, with $[q]=(p-3)/2$. As before $[E_p^0(x)=x]$ and $[E^{n+1}(x)=E(E^n(x))$. Furthermore, note that studying $3n+p$ on the negative integers amounts to studying $3n-p$ on the positive integers.

The $3n+p$ numbers come in three classes: with $p \in 1+6\0$, with $p \in 3+6\0$ and $p \in 5+6\0$. This because $3n+p=3(n-2)+p+6$: $3n+p$ has the same image as $3n+p+6$, but with a shift of 2 in terms of $n$, which means a shift of 1 in terms of $[x]$. The following table for $3n+p, \; p \in \{1 + 6\N\}$ will clarify this. In this table, all elements are in $[x]$.

\begin{tabular}{rrrrrrrrrr}
		  3n+37 & 3n+31 & 3n+25 & 3n+19 & 3n+13 & 3n+7 & 3n+1 & \text{maps to} & [x]\\
			-6 & -5 & -4 & -3 & -2 & -1 &  1  & ~ & 1\\
		  -5 & -4 & -3 & -2 & -1 &  1 &  2  & ~ & 3\\
			-4 & -3 & -2 & -1 &  1 &  2 &  3  & ~ & 1\\
		  -3 & -2 & -1 &  1 &  2 &  3 &  4  & ~ & 6\\
			-2 & -1 &  1 &  2 &  3 &  4 &  5  & ~ & 4\\
			-1 &  1 &  2 &  3 &  4 &  5 &  6  & ~ & 9\\
			 1 &  2 &  3 &  4 &  5 &  6 &  7  & ~ & 3\\
			 2 &  3 &  4 &  5 &  6 &  7 &  8  & ~ & 12\\
			 3 &  4 &  5 &  6 &  7 &  8 &  9  & ~ & 7\\
			 4 &  5 &  6 &  7 &  8 &  9 &  10 & ~ & 15\\
			 $\vdots$ & $\vdots$ & $\vdots$ & $\vdots$ & $\vdots$ & $\vdots$ & $\vdots$ & ~ & $\vdots$ 
\end{tabular}

\vspace{1cm}

Now the table for $3n+p, \; p \in \{-1 + 6\N\}$, again all in $[x]$:

\begin{tabular}{rrrrrrrrrr}
		  3n+35 & 3n+29 & 3n+23 & 3n+17 & 3n+11 & 3n+5 & 3n-1 & \text{maps to} & x\\
			-6 & -5 & -4 & -3 & -2 & -1 & 1  & ~ & 1\\
			-5 & -4 & -3 & -2 & -1 &  1 & 2  & ~ & 1\\
			-4 & -3 & -2 & -1 &  1 &  2 & 3  & ~ & 4\\
		  -3 & -2 & -1 &  1 &  2 &  3 & 4  & ~ & 3\\
			-2 & -1 &  1 &  2 &  3 &  4 & 5  & ~ & 7\\
			-1 &  1 &  2 &  3 &  4 &  5 & 6  & ~ & 1\\
			 1 &  2 &  3 &  4 &  5 &  6 & 7  & ~ & 10\\
			 2 &  3 &  4 &  5 &  6 &  7 & 8  & ~ & 6\\
			 3 &  4 &  5 &  6 &  7 &  8 & 9  & ~ & 13\\
			 4 &  5 &  6 &  7 &  8 &  9 & 10 & ~ & 4\\
			$\vdots$ & $\vdots$ & $\vdots$ & $\vdots$ & $\vdots$ & $\vdots$ & $\vdots$ & ~ & $\vdots$
\end{tabular}

\vspace{1cm}

Notice that in both tables, there is no ``$0^{th}$'' positive integer, so 1 and -1 are adjacent.

From equation (\ref{Generalization}), position mappings $G_p$ analogue to $F$ can be derived. For the first table, for $3n+p, \; p \in \{1 + 6\N\}$, these are:

$G_7:\N \rightarrow 1+3\0 \cup 3+3\0$, such that
\begin{equation}
\label{G7}
\begin{cases}
		G_7(E_7^n([1+2m])) = [3+3m]|(m,n) \in \0^2, \\
		G_7(E_7^n([4+4m])) = [4+3m]|(m,n) \in \0^2, \\
		G_7(E_7^n([2]) = [1]|n \in \0
\end{cases}
\end{equation}

$G_{13}:\N \rightarrow 1+3\0 \cup 3+3\0$, such that
\begin{equation}
\label{G13}
\begin{cases}
		G_{13}(E_{13}^n([2+2m])) = [6+3m]|(m,n) \in \0^2, \\
		G_{13}(E_{13}^n([3+4m])) = [4+3m]|(m,n) \in \0^2, \\
		G_{13}(E_{13}^n([1]) = [1]|n \in \0\\
		G_{13}(E_{13}^n([5]) = [3]|n \in \0	
\end{cases}
\end{equation}

$G_{19}:\N \rightarrow 1+3\0 \cup 3+3\0$, such that
\begin{equation}
\label{G19}
\begin{cases}
		G_{19}(E_{19}^n([1+2m])) = [6+3m]|(m,n) \in \0^2, \\
		G_{19}(E_{19}^n([2+4m])) = [4+3m]|(m,n) \in \0^2, \\
		G_{19}(E_{19}^n([8]) = [1]|n \in \0\\
		G_{19}(E_{19}^n([4]) = [3]|n \in \0	
\end{cases}
\end{equation}

$G_{25}:\N \rightarrow 1+3\0 \cup 3+3\0$, such that
\begin{equation}
\label{G25}
\begin{cases}
		G_{25}(E_{25}^n([2+2m])) = [9+3m]|(m,n) \in \0^2, \\
		G_{25}(E_{25}^n([1+4m])) = [4+3m]|(m,n) \in \0^2, \\
		G_{25}(E_{25}^n([7]) = [1]|n \in \0\\
		G_{25}(E_{25}^n([3]) = [3]|n \in \0\\
		G_{25}(E_{25}^n([11]) = [6]|n \in \0
\end{cases}
\end{equation}

$G_{31}:\N \rightarrow 1+3\0 \cup 3+3\0$, such that
\begin{equation}
\label{G31}
\begin{cases}
		G_{31}(E_{31}^n([1+2m])) = [9+3m]|(m,n) \in \0^2, \\
		G_{31}(E_{31}^n([4+4m])) = [7+3m]|(m,n) \in \0^2, \\
		G_{31}(E_{31}^n([6]) = [1]|n \in \0\\
		G_{31}(E_{31}^n([2]) = [3]|n \in \0\\
		G_{31}(E_{31}^n([14]) = [4]|n \in \0\\
		G_{31}(E_{31}^n([10]) = [6]|n \in \0
\end{cases}
\end{equation}

$G_{37}:\N \rightarrow 1+3\0 \cup 3+3\0$, such that
\begin{equation}
\label{G37}
\begin{cases}
		G_{37}(E_{37}^n([2+2m])) = [12+3m]|(m,n) \in \0^2, \\
		G_{37}(E_{37}^n([3+4m])) = [7+3m]|(m,n) \in \0^2, \\
		G_{37}(E_{37}^n([5]) = [1]|n \in \0\\
		G_{37}(E_{37}^n([1]) = [3]|n \in \0\\
		G_{37}(E_{37}^n([13]) = [4]|n \in \0\\
		G_{37}(E_{37}^n([9]) = [6]|n \in \0\\
		G_{37}(E_{37}^n([17]) = [9]|n \in \0
\end{cases}
\end{equation}

The reader will notice the pattern here: every time the $i$ in each of the restricted domains $G_p([i+2m]), \; i\in\{1,2\}$ and $G_p([i+4m]), \; i\in\{1,2,3,4\}$ loops to become 2, respectively 4, it releases a new restricted domain. Because of this, it is no longer possible to make statements of the kind ``of any and every 2 consecutive elements of $\N$, exactly 1 is mapped through $F_1$'' or ``of any and every 4 consecutive elements of $\N$, exactly 1 is mapped through $F_2$''. Rather, $G_p, \; p\neq1$ consist of multiple restricted domains that are going to form their own period restricted domains. Consider for instance $G_{37}$ above: in the lowest 5 lines we recognize:
\begin{equation}
\begin{cases}
		G_{37}(E_{37}^n([1 + 8m]) = [3+3m]|m \in \{0,1,2\}, n \in \0\\
		G_{37}(E_{37}^n([5 + 8m]) = [1+3m]|m \in \{0,1\}, n \in \0\\
\end{cases}
\end{equation}
These map into the same range, but the interval of their domains is different. Also, if we continue each of these periodic subsets, we find equivalents: $[1,9,17]$ do not have a lower equivalent, but $[17+8=25=E_{37}(2)]$, while $[5, 13]$ do not have a lower equivalent, but $[13+8=21=E_{37}(1)]$.

For $3n+p, \; p \in \{-1 + 6\N\}$ we have:

$G_{-1}:\N \rightarrow 1+3\0 \cup 3+3\0$, such that
\begin{equation}
\label{G-1}
\begin{cases}
		G_{-1}(E_{-1}^n([1+2m])) = [1+3m]|(m,n) \in \0^2 \\
		G_{-1}(E_{-1}^n([4+4m])) = [3+3m]|(m,n) \in \0^2
\end{cases}
\end{equation}

$G_5:\N \rightarrow 1+3\0 \cup 3+3\0$, such that
\begin{equation}
\label{G5}
\begin{cases}
		G_5(E_5^n([2+2m])) = [4+3m]|(m,n) \in \0^2 \\
		G_5(E_5^n([3+4m])) = [3+3m]|(m,n) \in \0^2 \\
		G_5(E_5^n([1]) = [1]|n \in \0
\end{cases}
\end{equation}

$G_{11}:\N \rightarrow 1+3\0 \cup 3+3\0$, such that
\begin{equation}
\label{G11}
\begin{cases}
		G_{11}(E_{11}^n([1+2m])) = [4+3m]|(m,n) \in \0^2 \\
		G_{11}(E_{11}^n([2+4m])) = [3+3m]|(m,n) \in \0^2 \\
		G_{11}(E_{11}^n([4]) = [1]|n \in \0
\end{cases}
\end{equation}

$G_{17}:\N \rightarrow 1+3\0 \cup 3+3\0$, such that
\begin{equation}
\label{G17}
\begin{cases}
		G_{17}(E_{17}^n([2+2m])) = [7+3m]|(m,n) \in \0^2 \\
		G_{17}(E_{17}^n([1+4m])) = [3+3m]|(m,n) \in \0^2 \\
		G_{17}(E_{17}^n([3]) = [1]|n \in \0 \\
		G_{17}(E_{17}^n([7]) = [4]|n \in \0
\end{cases}
\end{equation}

$G_{23}:\N \rightarrow 1+3\0 \cup 3+3\0$, such that
\begin{equation}
\label{G23}
\begin{cases}
		G_{17}(E_{17}^n([1+2m])) = [7+3m]|(m,n) \in \0^2 \\
		G_{17}(E_{17}^n([4+4m])) = [6+3m]|(m,n) \in \0^2 \\
		G_{17}(E_{17}^n([2]) = [1]|n \in \0 \\
		G_{17}(E_{17}^n([6]) = [4]|n \in \0 \\
		G_{17}(E_{17}^n([10]) = [3]|n \in \0
\end{cases}
\end{equation}

We a pattern similar to that found before. In particular, notice that for $3n-1$ numbers, the cycle $\{3,4\}$ is not in strings and $[2=E(1)]$, while for $3n+5$ numbers the loop $\{1\}$, non-trivial here, is not in strings. It seems too much of a coincidence that seeming to have a string partition seems to coincide with the conjecture seeming to be true. How is this for $3n+p, \; p \in \{3 + 6\N\}$?

Although $3n+p, \; p \in \{3 + 6\N\}$ is usually overlooked, there is not really a reason to. The situation for these mappings is different, though, because for p=3 the range is $2+3\0$. For $p=3$, $q=(3-3)/2=0$, so $E_3: [\N \rightarrow 4\N]$ such that
\begin{equation}
\label{E3}
E_3([x])=4[x]
\end{equation}
Then,
$G_3:\N \rightarrow 2+3\0$, such that
\begin{equation}
\label{G3}
\begin{cases}
		G_3(E_3^n([1+2m])) = [2+3m]|(m,n) \in \0^2 \\
		G_3(E_3^n([2+4m])) = [2+3m]|(m,n) \in \0^2
\end{cases}
\end{equation}
It is tempting but inconsistent to let $E_3([x])=2[x]$ and write
\begin{equation}
		G_3(E_3^n([1+2m])) = [2+3m]|(m,n) \in \0^2 
\end{equation}
Yet we retain equations (\ref{E3}) and (\ref{G3}) and the following exposition clarifies why.

\begin{tabular}{rrrrrrrrrr}
		  3n+39 & 3n+33 & 3n+27 & 3n+21 & 3n+15 & 3n+9 & 3n+3 & \text{maps to} & x\\
			-6 & -5 & -4 & -3 & -2 & -1 & 1  & ~ & 2\\
			-5 & -4 & -3 & -2 & -1 &  1 & 2  & ~ & 2\\
			-4 & -3 & -2 & -1 &  1 &  2 & 3  & ~ & 5\\
		  -3 & -2 & -1 &  1 &  2 &  3 & 4  & ~ & 2\\
			-2 & -1 &  1 &  2 &  3 &  4 & 5  & ~ & 8\\
			-1 &  1 &  2 &  3 &  4 &  5 & 6  & ~ & 5\\
			 1 &  2 &  3 &  4 &  5 &  6 & 7  & ~ & 11\\
			 2 &  3 &  4 &  5 &  6 &  7 & 8  & ~ & 2\\
			 3 &  4 &  5 &  6 &  7 &  8 & 9  & ~ & 14\\
			 4 &  5 &  6 &  7 &  8 &  9 & 10 & ~ & 8\\
			$\vdots$ & $\vdots$ & $\vdots$ & $\vdots$ & $\vdots$ & $\vdots$ & $\vdots$ & ~ & $\vdots$
\end{tabular}

\vspace{1cm}

From these, we derive the following mappings:

$G_3:\N \rightarrow 2+3\0$, such that
\begin{equation}
\begin{cases}
		G_3(E_3^n([1+2m])) = [2+3m]|(m,n) \in \0^2 \\
		G_3(E_3^n([2+4m])) = [2+3m]|(m,n) \in \0^2
\end{cases}
\end{equation}

$G_9:\N \rightarrow 2+3\0$, such that
\begin{equation}
\label{G9}
\begin{cases}
		G_9(E_9^n([2+2m])) = [5+3m]|(m,n) \in \0^2 \\
		G_9(E_9^n([1+4m])) = [2+3m]|(m,n) \in \0^2 \\
		G_9(E_9^n([3])) = [2]|n \in \0
\end{cases}
\end{equation}

$G_{15}:\N \rightarrow 2+3\0$, such that
\begin{equation}
\label{G15}
\begin{cases}
		G_{15}(E_{15}^n([1+2m])) = [5+3m]|(m,n) \in \0^2 \\
		G_{15}(E_{15}^n([4+4m])) = [5+3m]|(m,n) \in \0^2 \\
		G_{15}(E_{15}^n([2])) = [2]|n \in \0 \\
		G_{15}(E_{15}^n([6])) = [2]|n \in \0
\end{cases}
\end{equation}

$G_{21}:\N \rightarrow 2+3\0$, such that
\begin{equation}
\label{G21}
\begin{cases}
		G_{21}(E_{21}^n([2+2m])) = [8+3m]|(m,n) \in \0^2 \\
		G_{21}(E_{21}^n([3+4m])) = [5+3m]|(m,n) \in \0^2 \\
		G_{21}(E_{21}^n([1])) = [2]|n \in \0 \\
		G_{21}(E_{21}^n([5])) = [2]|n \in \0
		G_{21}(E_{21}^n([9])) = [5]|n \in \0
\end{cases}
\end{equation}

$G_{27}:\N \rightarrow 2+3\0$, such that
\begin{equation}
\label{G27}
\begin{cases}
		G_{27}(E_{27}^n([1+2m])) = [8+3m]|(m,n) \in \0^2 \\
		G_{27}(E_{27}^n([2+4m])) = [5+3m]|(m,n) \in \0^2 \\
		G_{27}(E_{27}^n([4])) = [2]|n \in \0 \\
		G_{27}(E_{27}^n([8])) = [5]|n \in \0
		G_{27}(E_{27}^n([12])) = [2]|n \in \0 \\
\end{cases}
\end{equation}

$G_{33}:\N \rightarrow 2+3\0$, such that
\begin{equation}
\label{G33}
\begin{cases}
		G_{33}(E_{33}^n([2+2m])) = [11+3m]|(m,n) \in \0^2 \\
		G_{33}(E_{33}^n([1+4m])) = [5+3m]|(m,n) \in \0^2 \\
		G_{33}(E_{33}^n([3])) = [2]|n \in \0 \\
		G_{33}(E_{33}^n([7])) = [5]|n \in \0 \\
		G_{33}(E_{33}^n([11])) = [2]|n \in \0 \\
		G_{33}(E_{33}^n([15])) = [8]|n \in \0
\end{cases}
\end{equation}

As above, if we continue the newly formed restricted periodic subsets, we encounter equivalents. E.g., for $G_{33}$: if we continue $[3,7,11,15]$ we find $19=E_{33}(1)$: the periodicity is broken.

Because the mapping for $3n+3$ numbers (equation (\ref{G3})) is two-to-one, strings take a somewhat different shape (Figure 2). Every position that is not the head has a co-tail that is not mapped to, while every position that is mapped to, is mapped to exactly twice. I have simulated these strings for the first $10^9$ positions, and found no exception.

\begin{figure}[h]
\includegraphics[width=.6\textwidth]{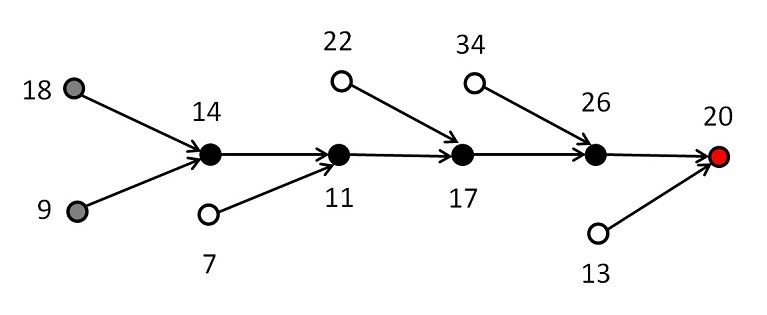}
\caption{Strings for $3n+3$ numbers: every position that is no a head has a co-tail, exactly half or double, while every position that is mapped to is mapped to exactly twice.}
\label{Figure 2}
\end{figure}

And yet, it seems that ``a Collatz conjecture for $3n+3$ numbers'' holds: every trajectory ends up in the `trivial-loop-plus-one' for $p=3$ (as simulated on Klaas IJntema's Collatz Calculation Center). Simultaneously, only the $3n+1$ and $3n+3$ numbers seem to have a string partition. Hence, I suggest that this perfect regularity for $3n+1$, which does not exist for other members of generalization $3n+p$, except for $3n+3$, for which the conjecture seems to be true too, could lead the way to a proof.

\section{References}
Lagarias, J.C. (2010) The 3x + 1 problem: an overview. In: The ultimate challenge: the 3x + 1 problem. Pages 3-29. American Mathematical Society, Providence, RI.
Pickover, C.A. (2009) The m$\alpha$th $\beta$ook. Sterling, New York.
Wirsching, G.J. (1998) The Dynamical System Generated by the 3n+1 Function. Springer.

\section{Appendix 1. Periodicity and sampling}
Sets and sequences can be \emph{periodic}. For sequences, this means that identical elements are found at fixed distances. For instance, $(1,2,1,2,\ldots)$ is a periodic sequence. For sets, this means that elements of a periodic set are found within $\N$ at a fixed distance. For instance, $\{11+64\0\}=\{11,75,139,\ldots\}$ is a periodic set. The fixed distance between consecutive elements, in the sequence itself for periodic sequence, in $\N$ for periodic sets, is defined to be the `interval', while the first (lowest) element of a periodic set or sequence is called the `intercept'.
\begin{definition}
\label{Interval}
An `interval' is the fixed distance between two consecutive identical elements in a periodic sequence, or the fixed distance in $\N$ between two consecutive elements of a periodic set.
\end{definition}
\begin{definition}
\label{Intercept}
An `intercept' is the first element of a periodic set or sequence.
\end{definition}
For instance, of the periodic set $\{11+64\0\}$, the intercept is 11 while the interval is 64. Likewise, the intercept of $\{188+64\0\}$ is 188, while the interval is 64. In some (but not all) cases in this paper the intervals are simply the modulus of some progression.

\begin{remark}
The periodicity here refers to periodicity in the \emph{ensemble}. Periodicity could also exist in \emph{time}. Indeed, the Collatz conjecture states that any natural number eventually yields the natural numbers $1,4,2,1,4,2,1,\ldots$: periodicity over time. On the other hand, if it is true that every second element of $\N$ has some property, then there is periodicity over the ensemble, in this case $\N$. This paper revolves around periodicity in the ensemble.
\end{remark}

\begin{remark}
There may be more than one interval in a periodic set or sequence (indeed, there may be infinitely many). For instance, in the sequence $(1,2,1,3,1,2,1,4,1,\ldots)$ there is an interval of 2 between elements 1, an interval of 4 between elements 2, and so forth.
\end{remark}

Given a sequence $(q_k)_{k=1}^{\infty}$, a sub-sequence could be sampled from $(q_k)_{k=1}^{\infty}$, taking elements at a fixed distance, henceforth called `period' $p$. If a sub-sequence is sampled with period $p$ from a periodic sequence or set, then there are two fixed distances involved: the interval of the periodic set or sequence, and the period with which a subset or sub-sequence is sampled. Hence the distinctive terminology:
\begin{definition}
\label{Period}
A `period' is a fixed step size by which elements are sampled from some set or sequence.
\end{definition}

Sampling from $(q_k)_{k=1}^{\infty}$ with period $p$ yields $(q_{kp+c})_{k=0}^{\infty}$, $q_c$ being the first element of $(q_k)_{k=1}^{\infty}$ that is sampled. If $(q_k)_{k=1}^{\infty}$ is a periodic sequence with interval(s) co-prime with the sampling period, what would be the result?

\begin{lemma}
\label{Coprime}
Let $(q_k)_{k=1}^{\infty}$ be a periodic sequence in which the intervals are positive integer powers of $r \in \N$, so $r^1, r^2, \ldots$. Let $p$ be a sampling period co-prime with the interval(s) of $(q_k)_{k=1}^{\infty}$. Then the interval(s) of $(q_{kp+c})_{k=0}^{\infty}$ are the same as the interval(s) in $(q_k)_{k=1}^{\infty}$.
\end{lemma}
\begin{proof}
Take some element $q_n$ of $(q_k)_{k=1}^{\infty}$. There are identical elements at distance $r^m$ for some $m \in \N$: $q_n, q_{n+r^m}, q_{n+2r^m}, \ldots$. Sampling with period $p$ starting at $q_n$ means that the elements $q_n, q_{n+p}, q_{n+2p}, \ldots$ are sampled. Where the sequences $q_n, q_{n+r^m}, q_{n+2r^m}, \ldots$ and $q_n, q_{n+p}, q_{n+2p}, \ldots$ intersect, i.e., at $q_n,q_{n+r^mp},q_{n+2r^mp},q_{n+3r^mp},\ldots$, identical elements are sampled. This means that $r^m$ steps of size $p$ are taken between two sampled identical elements. Since $(q_{kp+c})_{k=0}^{\infty}$ is sampled with step sizes $p$, in $(q_{kp+c})_{k=0}^{\infty}$ the interval between two consecutive identical elements is $r^m$, just as in the original sequence $(q_{k})_{k=0}^{\infty}$. Finally, because the interval between two identical elements is the same as in the original sequence, by exclusion (\emph{something} is sampled at the interjacent points), the interval between all two consecutive identical elements is as in the original sequence. This completes the proof.
\end{proof}

\section{Appendix 2. The proof of $F$.}
The accelerated Collatz map, that sends odd positive integers to odd positive integers, is:
\begin{equation}
\widetilde{C}(n):=\frac{3n+1}{2^j},
\label{AccelCollatz}
\end{equation}
where $2^j$ is the largest power of 2 that divides $3n+1$, with $n \in \D$, where $\D=1,3,5,\ldots$, the odd positive integers. Hence, I use the transformation $g: \D \rightarrow \N$, $g$ being defined as
\begin{equation}
g(n):=\frac{n+1}{2}.
\label{g}
\end{equation}

Conjugating $\widetilde{C}$ through $g$ yields mapping
\begin{equation}
F([x])=g(\widetilde{C}(g^{-1}([x]))),
\end{equation}
which is $F:\N \rightarrow 1+3\0 \cup 3+3\0$, such that
\begin{equation}
\begin{cases}
\label{F}
F(E^n([2+2m]))=[3+3m]|(m,n) \in \0^2, \\
F(E^n([1+4m]))=[1+3m]|(m,n) \in \0^2,
\end{cases}
\end{equation}
Here, $\0=0 \cup \N$ and $E: \N \rightarrow 3+4\0$ such that
\begin{eqnarray}
\label{E}
E([x]):&=&4[x]-1, \\
\label{E0}
E^0([x])&=& [x], \\
\label{En}
E^n([x])&=&E(E^{n-1}([x])).
\end{eqnarray}
The derivation of this mapping and the facts that the domain of $F$ is $[\N]$ while the range is $[1+3\0 \cup 3+3\0]$ are now proven in the following Lemmas (\ref{F1}) through (\ref{Domain}).

\begin{lemma}
\label{F1}
For $m \geq 0$,
\begin{equation}
F([2+2m])=[3+3m]
\end{equation}
\end{lemma}
\begin{proof}
The enumerated odd natural number $[x = 2+2m]$ corresponds to odd natural number $n=3+4m$. The Collatz iteration gives $3+4m \mapsto 10+12m \mapsto 5+6m$. This number is odd and corresponds to enumerated odd natural number $[3+3m]$.
\end{proof}

\begin{lemma}
\label{F2}
For $m \geq 0$,
\begin{equation}
F([1+4m])=[1+3m]
\end{equation}
\end{lemma}
\begin{proof}
The enumerated odd natural number $[x = 1+4m]$ corresponds to odd natural number $n=1+8m$. The Collatz iteration gives $1+8m \mapsto 4+24m \mapsto 2+12m \mapsto 1+6m$. This number is odd and corresponds to enumerated odd natural number $[1+3m]$.
\end{proof}

\begin{lemma}
\label{Ex}
The mapping (for $x \geq 1$)
\begin{equation}
E([x]):=[4x-1]
\end{equation}
has the property that
\begin{equation}
F(E([x]))=F([x]).
\end{equation}
\end{lemma}
\begin{proof}
Enumerated odd natural number $[x]$ corresponds to the odd natural number $n=2x-1$, which has the Collatz map image
\begin{equation}
2x-1 \mapsto 6x-2 \mapsto 3x-1 \; (a number).
\end{equation}
For $x \geq 1$ the enumerated odd natural number $[4x-1]$ corresponds to the odd natural number $8x-3$, which has the Collatz map image 
\begin{equation}
8x-3 \mapsto 24x-8 \mapsto 12x-4 \mapsto 6x-2 \mapsto 3x-1 \; (a number).
\end{equation}
As the Collatz map images match up to this point, the next odd number in both iterations is the same, and Lemma (\ref{E}) follows.
\end{proof}
Since equations (\ref{E0}) and (\ref{En}) generalize $E([x])$ (equation (\ref{E})), Lemmas \ref{F1}, \ref{F2} and \ref{Ex} together give $F$.

I now prove that $F: [\N \rightarrow 1+3\0 \cup 3+3\0]$.
\begin{lemma}
\label{Range}
$F$ does not map to $[2+3\0]$.
\end{lemma}
\begin{proof}
Enumerated odd natural numbers $[2+3\0]$ correspond to odd natural numbers $[3+6\0]$. These are multiples of 3, i.e., their residual in modulus 3 equals 0. If these odd natural numbers are multiplied by 2 an arbitrary number of times, the resulting even natural numbers are still in $[3+6\0]$ (and the modulus 3 residual still 0). No multiple of 3 can be obtained by multiplying a number by 3 and adding 1, and Lemma (\ref{Range}) follows.
\end{proof}

Lemma \ref{Range} refers to the range of $F$. $[\N]$ has a partition $[1+3\0 \cup 2+3\0 \cup 3+3\0]$, seen immediately from the fact that $1,2,3$ are the residual classes modulo 3. It turns out that of every and any three consecutive elements of $[\N]$, exactly two are in the range of $F$.

As for the domain of $F$, if it can be proven that there exists a partition of $[\N]$ identical to the domain of $F$ (equation (\ref{F})), i.e. $[E^{\0}(2+2\0) \cup E^{\0}(1+4\0)]$, $E$ defined in equations (\ref{E}-\ref{En}), then too it is proven that the domain of $F$ is $[\N]$. This assures that not only the mapping $F$ indeed follows from the Collatz map, but that \emph{all} Collatz maps of the natural numbers are represented in $F$. Hence the following Lemma.

\begin{lemma}
\label{Domain}
$[\N]$ has a partition $[E^{\0}(2+2\0) \cup E^{\0}(1+4\0)]$, $E$ defined in equations (\ref{E}-\ref{En}).
\end{lemma}
\begin{proof}
The map $[x] \rightarrow [4x-1]$ has as output a position that is 3 (mod 4). Every integer 3 (mod 4) has a unique preimage under $E(\cdot)$. If this preimage is 3 (mod 4), take the preimage of it. This can be continued until an element not 3 (mod 4) is reached. Since the position decreases at each step, halting will happen. This completes the proof.
\end{proof}

\end{document}